\documentclass[graybox]{svmult}
\usepackage{mathptmx}       
\usepackage{helvet}         
\usepackage{courier}        
\usepackage{type1cm}        
\usepackage{makeidx}         
\usepackage{graphicx}        
\usepackage[all]{xy}
\usepackage{multicol}        
\usepackage[bottom]{footmisc}
\usepackage{tikz}
\usepackage{tikz-cd}
\usetikzlibrary{arrows,automata}
\usetikzlibrary{shapes,snakes,backgrounds}
\usepackage{amsfonts}
\usepackage{amssymb}
\usepackage{amscd}
\usepackage{amstext}
\usepackage{amsmath}
\usepackage{pstricks}
\usepackage{enumerate}
\usepackage{hyperref}

\def\abs#1{|#1|}

\def\shuffle{\mathop{_{^{\sqcup\!\sqcup}}}}

\def\adots{\mathinner{\mkern2mu\raise1pt\hbox{.}
\mkern3mu\raise4pt\hbox{.}\mkern1mu\raise7pt\hbox{.}}}

\def\up#1{\raise 1ex\hbox{\footnotesize#1}}

\def\H{\mathcal{H}}

\def\span{\mathop\mathrm{span}\nolimits}

\def\path{\rightsquigarrow}

\def\Lyn{{\mathcal Lyn}}

\def\E{{\mathbb E}}

\def\C{{\mathbb C}}
\def\R{{\mathbb R}}

\def\Q{{\mathbb Q}}

\def\Q{{\mathbb Q}}

\def\H{\mathrm{H}}

\renewcommand{\Q}{{\mathbb Q}}
\renewcommand{\R}{{\mathbb R}}
\renewcommand{\C}{{\mathbb C}}

\newcommand{\calD}{{\mathcal D}}
\newcommand{\calH}{{\mathcal H}}

\newcommand{\calL}{{\mathcal L}}
\newcommand{\calX}{{\mathcal X}}

\newcommand{\calP}{{\mathcal P}}

\newcommand{\Li}{\operatorname{Li}}

\def\Lyn{\mathcal Lyn}

\def\calG{\mathcal{G}}

\def\H{\mathrm{H}}

\def\abs#1{|#1|}
\def\scal#1#2{\langle #1\mid#2 \rangle}
\def\ncp#1#2{#1\langle #2\rangle}
\def\ncs#1#2{#1\langle \!\langle #2\rangle \!\rangle}
\def\path{\rightsquigarrow}

\def\shuffle{\mathop{_{^{\sqcup\!\sqcup}}}} 
\def\conc{\mathop{\tt conc}} 

\def\calG{{\mathcal G}}
\def\calH{{\mathcal H}}
\def\calL{{\mathcal L}}
\def\calO{{\mathcal O}}

\def\calX{{\mathcal X}}

\def\scal#1#2{\langle #1 | #2 \rangle}

\def\ncs#1#2{#1\langle\langle #2\rangle\rangle}

\def\ncp#1#2{#1\langle #2\rangle}

\def\Li{\mathrm{Li}}

\gdef\stuffle{\;%
  \setlength{\unitlength}{0.0125cm}%
  \begin{picture}(20,10)(220,580)
  \thinlines
  \put(220,592){\line( 0,-1){ 10}}
  \put(220,582){\line( 1, 0){ 20}}
  \put(240,582){\line( 0, 1){ 10}}
  \put(230,592){\line( 0,-1){ 10}}
  \put(225,587){\line( 1, 0){ 10}}
  \end{picture}\;
}

\def\rd{\triangleright}
\def\rg{\triangleleft}

\def\trl{\triangleleft}

\newcommand\rsmraise[1]{%
  \ifx#1\displaystyle .8\else
    \ifx#1\textstyle .8\else
      \ifx#1\scriptstyle .6\else
        .45%
      \fi
    \fi
  \fi}
\begingroup
\def\2#1{\ifnum#1<10 0\fi\the#1}
\xdef\isodayandtime{
{\2\day-\2\month-\the\year\space\2{\count0}:%
\2{\count2}}}
\endgroup

\begingroup
\def\2#1{\ifnum#1<10 0\fi\the#1}
\xdef\isodayandtime{
{\2\day-\2\month-\the\year\space\2{\count0}:%
\2{\count2}}}
\endgroup

\newcounter{per1}
\definecolor{MyDarkBlue}{rgb}{0,0.08,0.4}

\begin{document}

\title*{Hyperlogarithms: Functions on Free Monoids}
\author{V. Hoang Ngoc Minh}

\maketitle

\abstract{To factorize and to decompose the graphs of \textit{representative} functions on the free monoid $\calX^*$ (generated by the alphabet $\calX$) with values in the ring $A$ containing $\Q$, we examine various products of series (as concatenation, shuffle and its $\phi$-deformations) and co-products, which are such that their associated \textit{non graded} bialgebras are isomorphic, for $A$ is a field $K$, to the Sweedler's dual of the \textit{graded} noncommutative co-commutative $K$-bialgebra of polynomials.}

\section{Introduction}\label{intro}
Hopf algebras involve in algebraic geometry and topology, theory of algebraic groups and representation theory, in which representative functions with values in a field $K$ were investigated on a group $G$ \cite{abe,Hochschild}.
As in \cite{CartierPatras}, let us consider the following function on the monoid ${\cal X}^*$ generated by an alphabet ${\cal X}$
\begin{eqnarray}\label{f}
f:{\cal X}^*\longrightarrow K.
\end{eqnarray}

\begin{example}[\cite{CM}]\label{pos}
Polylogarithms (resp. harmonic sums) are holomorphic (resp. arithmetical) functions $\{\Li_{s_1,\cdots,s_r}\}_{s_1,\cdots,s_r\ge1,r\ge0}$ (resp. $\{\H_{s_1,\cdots,s_r}\}_{s_1,\cdots,s_r\ge1,r\ge0}$) defined, for any 
multiindex $(s_1,\ldots,s_r)$ in the free monoid $({\mathbb N}_{\ge1})^*$ generated by ${\mathbb N}_{\ge1}$, by
\begin{eqnarray*}
\Li_{s_1,\ldots,s_r}(z)=\sum_{n_1>\cdots>n_r>0}\frac{z^{n_1}}{n_1^{s_1}\cdots n_r^{s_r}}&\Big(\mbox{resp.}&
\H_{s_1,\ldots,s_r}(n)=\sum_{n\ge n_1>\cdots>n_r>0}\frac1{n_1^{s_1}\cdots n_r^{s_r}}\Big).
\end{eqnarray*}
Since $(s_1,\ldots,s_r)$ one-to-one corresponds to the word $x_0^{s_1-1}x_1\cdots x_0^{s_r-1}x_1$ (resp. $y_{s_1}\cdots\allowbreak y_{s_r}$) of the monoid $X^*$ (resp. $Y^*$) generated by $X=\{x_0,x_1\}$ (resp. $Y=\{y_k\}_{k\ge1}$) then $\Li_{\bullet}$ (resp. $\H_{\bullet}$) is a function, as well as on $({\mathbb N}_{\ge1})^*$ than on $X^*$ (resp. $Y^*$), to the ring of polylogarithms (resp. harmonic sums), in which $\Li_{s_1,\ldots,s_r}=\Li_{x_0^{s_1-1}x_1\ldots x_0^{s_r-1}x_1}$ (resp. $\H_{s_1,\ldots,s_r}=\H_{y_{s_1}\ldots y_{s_r}}$) and by convention, $\Li_{x_0}(z)$ stand for $\log(z)$ \cite{CM}. It turns out that $\Li_{\bullet}$ (resp. $\H_{\bullet}$) realizes an isomorphism between the shuffle (resp. quasi-shuffle) algebra of noncommutative polynomials $(\ncp{\Q}{X},\shuffle,1_{X^*})$ (resp. $(\ncp{\Q}{Y},\stuffle,1_{Y^*})$\footnote{In Section \ref{coproducts} below, $\stuffle$ will be considered as a deformation of $\shuffle$ over $\ncp{\mathbb C}{Y}$.}) and the algebra of polylogarithms (resp. harmonic sums) \cite{CM}. These holomorphic (resp. arithmetical) functions $\{\Li_w\}_{w\in X^*}$ (resp. $\{\H_w\}_{w\in Y^*}$) lie, in particular, as follows
\begin{eqnarray*}
(1-z)^{-1}{\Li_{x_0^{s_1-1}x_1\ldots x_0^{s_r-1}x_1}(z)}&=&
\sum_{n\ge0}\H_{y_{s_1}\ldots y_{s_r}}(n)z^n.
\end{eqnarray*}
\end{example}

The function $f$ in \eqref{f} is a \textit{representative} if and only if there is finitely many functions $\{f'_i,f''_i\}_{i\in I_{finite}}$ of $K^{\calX^*}$, which can be choosen to be \textit{representative} functions such that, for any $u$ and $v\in\calX^*$, one has \cite{CartierPatras}
\begin{eqnarray}\label{representative}
f(uv)=\sum_{i\in I_{finite}}f'_i(u)f''_i(v).
\end{eqnarray}
With the notations in \eqref{f}--\eqref{representative}, the coproduct of \textit{representative} function $f$ can be defined in duality with the product in $\calX^*$ (\textit{i.e.} the concatenation, denoted by $\tt conc$ and omitted when there is no ambiguity) as follows \cite{CartierPatras}
\begin{eqnarray}\label{Delta_f}
\forall u,v\in\calX^*,
&\Delta_{\conc}(f)(u\otimes v)=f(uv),
&\Delta_{\conc}(f)=\sum_{i\in I_{finite}}f'_i\otimes f''_i.
\end{eqnarray}
The graph of $f$ in \eqref{f}, viewed as a noncommutative generating series over $\calX$ and with coefficients in $K$, is described as follows
\begin{eqnarray}\label{ngs}
S=\sum_{w\in\calX^*}\scal{S}{w}w,&\mbox{where}&\scal{S}{w}=f(w).
\end{eqnarray}

Any series $S$ is defined as a function $\calX^*\longrightarrow K$ mapping $w$ to $\scal{S}{w}$, so-called \textit{coefficient} of $w$ in $S$, and its \textit{graph} is the infinite sum on $\{\scal{S}{w}w\}_{w\in{\cal X}^*}$ \cite{berstel}. It is \textit{rational} if and only if there is an interger $n$ and a triplet $(\nu,\mu,\eta)$, with $\nu\in M_{1,n}(K)$ and $\eta\in M_{n,1}(K)$ and $\mu:{\cal X}^*\longrightarrow M_{n,n}(K)$, such that $\scal{S}{w}=\nu\mu(w)\eta$ \cite{berstel}. The triplet $(\nu,\mu,\eta)$ is called \textit{linear representation}\footnote{By left or right shifts (see Definition \ref{dec1} below), minimization algorithms provide minimal linear representations of the smallest rank \cite{berstel}.} of rank $n$ of $S$ \cite{berstel} and the morphism $\mu$ is called \textit{linear representation} of the monoid ${\cal X}^*$. For any $0\le i\le n$, letting $G_i$ (resp. $D_i$) be a rational series admitting $(\nu,\mu,e_i)$ (resp. $({}^te_i,\mu,\eta)$), with $e_i\in M_{1,n}(A)$ and $\underset{\hskip1cm\uparrow i}{{}^te_i={\begin{matrix}(0&\ldots&0&1&0&\ldots&0)\end{matrix}}}$, as linear representation of rank $n$ and extending $\Delta_{\conc}$ over the $K$-algebra of series, one has \cite{CM}
\begin{eqnarray}\label{Delta_S}
\Delta_{\conc}(S)=\sum_{1\le i\le n}G_i\otimes D_i.
\end{eqnarray}
Hence, the function $f$ in \eqref{f} is representative if and only if the series $S$ in \eqref{ngs} is rational and is said to be representative (so do $G_i$ and $D_i$) \cite{legsloops}. By \eqref{Delta_f}--\eqref{Delta_S}, one also obtains (see also Definition \ref{rational}, Theorem \ref{KS} and Proposition \ref{PQ} below)
\begin{eqnarray}
\forall u,v\in\calX^*,\quad
\Delta_{\conc}(f)(u\otimes v)=&f(uv)&=\scal{S}{uv},\\
\scal{\Delta_{\conc}(S)}{u\otimes v}
=&\displaystyle\sum_{1\le i\le n}\scal{G_i}{u}\scal{D_i}{v}&=\sum_{1\le i\le n}f'_i(u)f''_i(v),\label{Delta_fi}\\
\nu\mu(u)e_i=f'_i(u)=&\scal{G_i}{u}\mbox{ and }\scal{D_i}{v}&=f''_i(v)={}^te_i\mu(v)\eta.
\end{eqnarray}

In order to express solutions of fuchian differential equations with hyperlogarihms \cite{PMB,Linz,ACA,CM}, representative series are viewed as noncommutative generating series of representative functions, on ${\cal X}^*$ with values in a ring $A$ containing $\mathbb Q$. It will be effectively factorized and decomposed within their associated $A$-bialgebras, basing on monoidal factorizations and extending the results (concerning shuffle and quasi-shuffle, and already obtained over $\Q$ or $\C$ in \cite{PMB,Linz,ACA,CM}) to study $\phi$-shuffle $A$-bialgebra of representative series over $\calX$.
For that, in the next sections, we will examine combinatorial aspects of various products (as concatenation, shuffle and its $\phi$-deformations denoted by $\shuffle_{\phi}$) and their coproducts, for which primitive and grouplike series will be characterized, by Proposition \ref{prim_gp}. Moreover, pairs of dual bases, for $\shuffle$ and for ${\shuffle}_{\phi}$ graded bialgebras, will be constructed to factorize diagonal series (see \eqref{diagonalX} and \eqref{diagonalY}) and then the representative series (see Corollary \ref{factorized}). For $A=K$, Sweedler's duals of the $\shuffle$ (resp. ${\shuffle}_{\phi}$) bialgebras of polynomials will be proved to be isomorphic to (non graded) bialgebras of representative series over $\calX$ with coefficients in $A$ (see Proposition \ref{PQ}, Theorem \ref{KS}, Corollary \ref{Sweedler}).

Ending this introduction, let us illustrate our purposes with the following linear differential equation, of order $n\ge0$ with coefficients $\{a_i\}_{0\le i\le n}$ in $\C(z)$,
\begin{eqnarray}
a_n(z)\partial_z^ny(z)+\cdots+a_1(z)\partial_zy(z)+a_0(z)=0,&\mbox{where}&\partial_z={d}/{dz},\label{diffeq}
\end{eqnarray}
putted in the form of linear dynamical system, with the observation $\lambda\in M_{1,n}(\C)$, the initial state $\eta\in M_{1,n}(\C)$, the rational inputs $(u_i)_{0\le i\le m}$ and the matrices $\{M_i\}_{0\le i\le m}$ in $M_{n,n}(\C)$, as follows (see \cite{ACA})
\begin{eqnarray}
\partial_zq=(M_0(q)u_0+\ldots+M_m(q)u_m)q,&q(z_0)=\eta,&y=\lambda q.\label{ds}
\end{eqnarray}

\begin{example}[hypergeometric equation, $m=1$]\label{Hypergeometric}
Let $t_0,t_1,t_2$ be parameters and
\begin{eqnarray*}
z(1-z)\partial_z^2y(z)+[t_2-(t_0+t_1+1)z]\partial_zy(z)-t_0t_1 y(z)=0.
\end{eqnarray*}
For ${}^t(q_1(z),q_2(z))={}^t(-y(z),(1-z)\partial_zy(z))$ and $u_0(z)=z^{-1}$ and $u_1(z)=(1-z)^{-1}$, this hypergeometric equation is represented by $\partial_zq=(M_0u_0+M_1u_1)q$, where
\begin{eqnarray*}
M_0=-\begin{pmatrix}0&0\cr t_0t_1&t_2\end{pmatrix}\mbox{ and }M_1=-\begin{pmatrix}0&1\cr0&t_2-t_0-t_1\end{pmatrix}
\in{\cal M}_{2,2}({\mathbb C}[t_0,t_1,t_2]).
\end{eqnarray*}
\end{example}

It is convenient (and possible) to separate the contribution of $(M_i)_{0\le i\le m}$ and that of the differential forms $(\omega_i)_{0\le i\le m}$, defined by $\omega_i=u_idz$, through the alphabet $X=\{x_i\}_{0\le i\le m}$ generating the monoid $(X^*,1_{X^*})$. Indeed, under convergence conditions \cite{ACA,fliess2}, $y$ (depending on $z_0$) is computed as follows 
\begin{eqnarray}\label{output}
y(z)=\sum_{w\in X^*}\nu\mu(w)\eta\alpha_{z_0}^z(w),&\mbox{where}&
\alpha_{z_0}^z(w)=\left\{\begin{array}{rlc}
1_{\calH(\Omega)}&\mbox{if}&w=1_{X^*},\cr
\displaystyle\int_{z_0}^z\omega_i(s)\alpha_{z_0}^{s}(v)&\mbox{if}&w=x_iv,
\end{array}\right.
\end{eqnarray} 
as a pairing of the generating series of (\ref{ds}) \cite{fliess2} and the Chen series \cite{chen}:
\begin{eqnarray}\label{ngsofds}
F=\sum_{w\in X^*}\nu\mu(w)\eta w
&\mbox{and}&C_{z_0\path z}=\sum_{w\in X^*}\alpha_{z_0}^z(w)w.
\end{eqnarray}

Hence, the system in \eqref{ds} is associated to the triplet $(\nu,\mu,\eta)$
and provides two  functions over the free monoid $X^*$, $\mu$ and $\alpha_{z_0}^z$. Moreover, the iterated integrals $\{\alpha_{z_0}^z(w)\}_{w\in X^*}$ (of $(\omega_i)_{0\le i\le m}$ and along the path $z_0\path z$ over a simply connected manifold $\Omega$) belong to the ring of holomorphic functions, $\calH(\Omega)$.

\begin{example}
By Examples \ref{pos}-\ref{Hypergeometric}, for $\omega_0(z)=z^{-1}dz$ and $\omega_1(z)=(1-z)^{-1}dz$ along $z_0\path z$ over $\widetilde{\C\setminus\{0,1\}}$, one has $\alpha_0^z(x_0^{s_1-1}x_1\cdots x_0^{s_r-1}x_1)=\Li_{s_1,\cdots,s_r}(z)\in\calH(\widetilde{\C\setminus\{0,1\}})$.
\end{example}

Generally, for any set of singularities $\sigma=\{s_i\}_{i\ge0}$ ($s_0=0$), let $\rho_i=s_i^{-1}$ and $\omega_i(z)=u_i(z)dz$, where $u_i(z)=(z-s_i)^{-1}=\rho_i(1-\rho_iz)^{-1}$. Suppose that
\begin{eqnarray}\label{assumption}
\mbox{if }i\neq j\mbox{ then }s_i\neq s_j(i,j\ge0)\mbox{ and }s_i=e^{\mathrm{i}\theta_i},\mbox{ with }\theta_i\in]-\pi,\pi[.
\end{eqnarray}
For $X=\{x_i\}_{i\ge0}$ and $Y=\{y_{s_k,\rho}\}_{k\ge1,\rho\in\sigma}$, let $\pi_Y:X^*(X\setminus\{x_0\})\longrightarrow Y^*$ be a $\conc$-morphims mapping $x_0^{s-1}x_i$ to $y_{s,\rho_i}$. For $w=x_0^{s_1-1}x_{i_1}\ldots x_0^{s_r-1}x_{i_r}$, $\alpha_0^z(w)\in\calH(\widetilde{\C\setminus\sigma})$ is a hyperlogarithm \cite{lappo,Linz,Todorov} (or Dirichlet function \cite{FPSAC95,FPSAC96}):
\begin{eqnarray}
&\alpha_0^z(w)
&=\Li_w(z)
=\sum_{n_1>\cdots>n_r>0}\frac{\rho_{i_1}^{n_1}\cdots\rho_{i_r}^{n_r}}{n_1^{s_1}\cdots n_r^{s_r}}z^{n_1}\label{polylogarithms}
\end{eqnarray}
and the following ratio yields an extended harmonic sum, as arithmetic function,
\begin{eqnarray}
\frac{\Li_w(z)}{1-z}=\sum_{n\ge0}\H_{\pi_Y w}(n)z^n,&\mbox{where}&
\H_{\pi_Y w}(n)
=\sum_{n\ge n_1>\cdots>n_r>0}\frac{\rho_{i_1}^{n_1}\cdots\rho_{i_r}^{n_r}}{n_1^{s_1}\cdots n_r^{s_r}}.\label{harmonic_sums}
\end{eqnarray}
Hence, $\Li_{\bullet}$ (resp. $\H_{\bullet}$) is a function on the free monoid $X^*$ (resp. $Y^*$) to the ring of hyperlogarithms (resp. extended harmonic sums) in which by convention, $\Li_{x_0}(z)$ stand for $\log(z)$ \cite{Linz}.
Moreover, by the assumption in \eqref{assumption}, the logarithms $\{\Li_x\}_{x\in X}$ are linearly free over $\Q[z,z^{-1},\{(1-\rho_iz)^{-1}\}_{1\le i\le m}]$ and it follows that, by Lemma 2.2 in \cite{PMB}, $\{\Li_w\}_{w\in X^*}$ are linearly free over $\Q[z,z^{-1},\{(1-\rho_iz)^{-1}\}_{1\le i\le m}]$ and then $\{\H_w\}_{w\in Y^*}$ are $\Q$-linearly free, as the Taylor coefficients of $\{(1-z)^{-1}\Li_{\pi_X w}(z)\}_{w\in Y^*}$.

For $(s_{i_1},\rho_{i_1})\neq(1,1)$, the following limits exist and coincide with
\begin{eqnarray}\label{extendedpolyzeta}
\zeta\Big({\rho_{i_1}\atop s_{i_1}}\cdots{\rho_{i_r}\atop s_{i_r}}\Big):=
\lim_{z\to0}\Li_w(z)=\lim_{n\to+\infty}\H_{\pi_Y w}(n)
=\sum_{n_1>\cdots>n_r>0}\frac{\rho_{i_1}^{n_1}\cdots\rho_{i_r}^{n_r}}{n_1^{s_1}\cdots n_r^{s_r}},
\end{eqnarray}
which is an extended polyzeta, \textit{i.e.} the $\zeta$ in \eqref{extendedpolyzeta} is a partial function on free monoid $\big({\sigma\atop{\mathbb N}_{\ge1}}\big)^*$ to $\mathbb R$. It can be, similarly to the ordinary $\zeta$ polymorphism \cite{CM}, realized as a polymorphism from the subalgebra of $(\ncp{\Q}{X},\shuffle,1_{X^*})$ (resp. $(\ncp{\Q}{Y},\shuffle_{\phi},1_{Y^*})$) to $\R$, where $\shuffle_{\phi}$ is a deformation of $\shuffle$, over $\ncp{\mathbb C}{Y}$ (see Section \ref{coproducts} below).

\begin{example}[coulored polylogarithms and coulored harmonic sums, \cite{legsloops}]
Let $\calO_m=\{\rho_i\}_{1\le i\le m}$, where $\rho_i=e^{\mathrm{i}\frac{2\pi}{m}i}$. Let $\omega_0(z)=z^{-1}{dz}$ and $\omega_i(z)=\rho_i(1-\rho_iz)^{-1}{dz}$, for $1\le i\le m$. Let $X=\{x_0,\cdots,x_m\}$ and $Y=\{y_{s_i,\rho}\}_{i\ge1,\rho\in{\cal O}_m}$. For any coulored multiindex $\big({\rho_{i_1}\atop s_1}\cdots{\rho_{i_r}\atop s_r}\big)\in\big({{\cal O}_m\atop{\mathbb N}_{\ge1}}\big)^*$ associated to $x_0^{s_1-1}x_{i_1}\cdots x_0^{s_r-1}x_{i_r}\in X^*(X\setminus\{x_0\})$ and to $y_{s_{i_1},\rho_{i_1}}\allowbreak\cdots y_{s_r,\rho_{i_r}}\in Y^*$, the iterated integral $\alpha_0^z(x_0^{s_1-1}x_{i_1}\cdots x_0^{s_r-1}x_{i_r})$ is the following coulored polylogarithm and the following ratio yields the so-called coulored harmonic sum
\begin{eqnarray*}
\Li_{x_0^{s_1-1}x_{i_1}\cdots x_0^{s_r-1}x_{i_r}}(z)
=\Li_{{\rho_{i_1}\atop s_1}\cdots{\rho_{i_r}\atop s_r}}(z)
&=&\sum_{n_1>\cdots>n_r>0}\frac{\rho_{i_1}^{n_1}\cdots\rho_{i_r}^{n_r}}{n_1^{s_1}\cdots n_r^{s_r}}{z^{n_1}},\\
(1-z)^{-1}\Li_{x_0^{s_1-1}x_{i_1}\cdots x_0^{s_r-1}x_{i_r}}(z)
&=&\sum_{n\ge0}\H_{y_{s_{i_1},\rho_{i_1}}\cdots y_{s_r,\rho_{i_r}}}(n)z^n,\cr
\H_{y_{s_{i_1},\rho_{i_1}}\cdots y_{s_r,\rho_{i_r}}}(n)
=\H_{{\rho_{i_1}\atop s_1}\cdots{\rho_{i_r}\atop s_r}}(n)
&=&\sum_{n\ge n_1>\cdots>n_r>0}\frac{\rho_{i_1}^{n_1}\cdots\rho_{i_r}^{n_r}}{n_1^{s_1}\cdots n_r^{s_r}}.
\end{eqnarray*}
Hence, $\Li_{\bullet}$ (resp. $\H_{\bullet}$) is a function as well as on the monoid $\big({{\cal O}_m\atop{\mathbb N}_{\ge1}}\big)^*$ than on $X^*$ (resp. $Y^*$) to the ring of coulored polylogarithms (resp. harmonic sums), \textit{i.e.}\footnote{Recall also that, by convention, $\Li_{x_0}(z)$ stand for $\log(z)$)} $\{\log(z)\}\cup\{\Li_{{\rho_{i_1}\atop s_1},\cdots,{\rho_{i_r}\atop s_r}}(z)\}_{{\rho_{i_1},\cdots,\rho_{i_r}\in\calO_m\atop s_1,\cdots,s_r\ge1,r\ge0}}$ (resp. $\{\H_{{\rho_{i_1}\atop s_1},\cdots,{\rho_{i_r}\atop s_r}(n)}\}_{{\rho_{i_1},\cdots,\rho_{i_r}\in\calO_m\atop s_1,\cdots,s_r\ge1,r\ge0}}$). For $(s_{i_1},\rho_{i_1})\neq(1,1)$, the following limits exist and coincide with the so-called coulored polyzeta \cite{legsloops}:
\begin{eqnarray*}
\left\{\begin{array}{rcc}
\zeta(x_0^{s_1-1}x_{i_1}\cdots x_0^{s_r-1}x_{i_r})&:=&
\lim\limits_{z\to1}\Li_{{\rho_{i_1}\atop s_1}\cdots{\rho_{i_r}\atop s_r}}(z)\\
\zeta(y_{s_{i_1},\rho_{i_1}}\cdots y_{s_r,\rho_{i_r}})&:=&
\lim\limits_{n\to+\infty}\H_{{\rho_{i_1}\atop s_1}\cdots{\rho_{i_r}\atop s_r}}(n)
\end{array}\right\}
=\zeta\Big({\rho_{i_1}\atop s_1}\cdots{\rho_{i_r}\atop s_r}\Big).
\end{eqnarray*}
This common limit, as  special value of coulored polylogarithm, is an iterated integral and satisfies shuffle relations. As limit of coulored harmonic sum for $n\to+\infty$, it satisfies also the coulored quasi-shuffle relations, induced by the coulored quasi-shuffle product, as a deformation of $\shuffle$, over $\ncp{\mathbb C}{Y}$, which is defined by $u\stuffle1_{Y^*}=1_{Y^*}\stuffle u=u$ and $(y_{s,\rho}u)\stuffle(y_{s',\rho'}u')=y_{s,\rho}(u\stuffle(y_{s',\rho'}u'))+y_{s',\rho'}((y_{s,\rho}u)\stuffle u')+y_{s+s',\rho\rho'}(u\stuffle u')$, for $y_{s,\rho},y_{s',\rho'}\in Y$ and $u,u'\in Y^*$ (see also Section \ref{coproducts} below).
\end{example}

\section{Various products of formal power series}\label{coproducts}
In all the sequel, unless explicitly stated, all tensor products will be considered over the ring $A$ containing $\Q$. Let $X=\{x_0,\cdots,x_m\},x_0\prec\cdots\prec x_m$ (resp. $Y=\{y_k\}_{k\ge1},y_1\succ y_2\succ\ldots$) generate the monoid $X^*$ (resp. $Y^*$) with respect to the concatenation, denoted by $\conc$ and omitted when there is no ambiguity. For all matters concerning $X$ or $Y$, a generic model noted $\calX$ is used to state their common combinatorial features.

For ${\calX}=X$ or $Y$ the corresponding monoids are equipped with length functions\footnote{For $X$ we consider the length of words (\textit{i.e.} $(w)=\ell(w)=\abs{w}$) and for $Y$ the length is given by the weight (\textit{i.e.} $(w)=\ell(y_{i_1}\ldots y_{i_n})=i_1+\ldots+i_n$).}, inducing a grading of $\ncp{A}{{\calX}}$ and $\ncp{\calL ie_{A}}{{\calX}}$ in free modules of finite dimensions.
The module $\ncp{A}{\calX}$ is endowed with the unital associative concatenation product and the unital associative commutative shuffle product, defined by $u\shuffle 1_{\calX^*}=1_{\calX^*}\shuffle u=u$ and $
xu\shuffle yv=x(u\shuffle yv)+y(xu\shuffle v)$, for $x,y\in\calX$ and $u,v\in\calX^*$ \cite{berstel}. As morphisms for $\conc$, the coproducts $\Delta_{\conc}$ and $\Delta_{\shuffle}$ are defined on letters $x$, by $\Delta_{\conc}(x)=\Delta_{\shuffle}(x)=1_{\calX^*}\otimes x+x\otimes1_{\calX^*}$.

By a Radford's theorem, the set of Lyndon words, denoted by $\Lyn\calX$, forms a pure transcendence basis of the algebra $(\ncp{A}{\calX},\shuffle,1_{\calX^*})$ \cite{reutenauer}. It is known that the enveloping algebra $\mathcal{U}(\ncp{\calL ie_{A}}{\calX})$ is isomorphic to the connected, graded and cocommutative bialgebra $\calH_{\shuffle}(\calX)=(\ncp{A}{\calX},\conc,1_{\calX^*},\Delta_{\shuffle})$, being equipped the linear basis $\{P_w\}_{w\in \calX^*}$ (expanded after the homogeneous basis $\{P_l\}_{l\in \Lyn\calX}$ of $\ncp{\calL ie_{A}}{\calX}$) and its graded dual basis $\{S_w\}_{w\in\calX^*}$ (containing the transcendence basis $\{S_l\}_{l\in\Lyn\calX}$ of the $\shuffle$-algebra) \cite{reutenauer}. Let $\calH_{\shuffle}^{\vee}(\calX)=(\ncp{A}{\calX},{\shuffle},1_{\calX^*},\Delta_{\conc})$ be the graded dual of $\calH_{\shuffle}(\calX)$. Then the diagonal series ${\calD}_{\calX}$ is factorized by \cite{reutenauer}
\begin{eqnarray}\label{diagonalX}
{\calD}_{\calX}:=\sum_{w\in\calX^*}w\otimes w=\sum_{w\in\calX^*}S_w\otimes P_w
=\prod_{l\in\Lyn\calX}^{\searrow}e^{S_l\otimes P_l}.
\end{eqnarray}

Additionally and similarly to the quasi-shuffle case \cite{VJM,CM}, $\ncp{A}{Y}$ is also equipped with the unital associative commutative  product, $\shuffle_{\phi}$, defined for any $u,v\in Y^*$ and $y_i,y_j\in Y$, by $u{\shuffle_{\phi}}1_{Y^*}=1_{Y^*}{\shuffle_{\phi}}u=u$ and
\begin{eqnarray}
y_iu{\shuffle_{\phi}}y_jv=y_i(u{\shuffle_{\phi}}y_jv)+y_j(y_iu{\shuffle_{\phi}}v)+\phi(y_i,y_j)(u{\shuffle_{\phi}}v),
&&\phi(y_i,y_j)=\sum_{i+j=k}\gamma_{i,j}^ky_k,
\end{eqnarray}
and its dual law, being a $\conc$-morphism, is given by
\begin{eqnarray}\label{Dstuffle}
\forall y_k\in Y,\Delta_{\shuffle_{\phi}}(y_k)=y_k\otimes 1_{Y^*}+1_{Y^*}\otimes y_k +\sum_{i+j=k}\gamma_{i,j}^ky_i\otimes y_j.
\end{eqnarray}
For $\mathrm{Prim}_{\shuffle_{\phi}}(Y)=\span_{A}\{\pi_1w\}_{w\in Y^*}$, where $\pi_1$ is the eulerian idempotent defined by
\begin{eqnarray}\label{pi_1}
\forall w\in Y^*,&&
\pi_1w=w+\sum_{k=2}^{(w)}\frac{(-1)^{k-1}}k\sum_{u_1,\ldots,u_k\in Y^+}\scal{w}{u_1{\shuffle_{\phi}}\ldots{\shuffle_{\phi}}u_k}u_1\ldots u_k,
\end{eqnarray}
the enveloping algebra $\mathcal{U}(\mathrm{Prim}_{\shuffle_{\phi}}(Y))$ is isomorphic to the connected, graded and cocommutative bialgebra $\calH_{\shuffle_{\phi}}(Y)=(\ncp{A}{Y},\conc,1_{Y^*},\Delta_{\shuffle_{\phi}})$ admitting $\calH_{\shuffle_{\phi}}^{\vee}(Y)=(\ncp{A}{Y},{\shuffle_{\phi}},1_{Y^*},\Delta_{\conc})$ as dual, in which the diagonal series ${\calD}_Y$ is factorized by
\begin{eqnarray}\label{diagonalY}
{\calD}_Y:=\sum_{w\in Y^*}w\otimes w=
\sum_{w\in Y^*}\Sigma_w\otimes\Pi_w
=\prod_{l\in\Lyn Y}^{\searrow}e^{\Sigma_l\otimes\Pi_l},
\end{eqnarray}
where $\{\Pi_w\}_{w\in Y^*}$ is the linear basis (expanded by PBW after the homogeneous in weight basis $\{\Pi_l\}_{l\in \Lyn Y}$ of $\mathrm{Prim}_{\shuffle_{\phi}}(Y)$) and $\{\Sigma_w\}_{w\in Y^*}$ is its dual basis (containing the pure transcendence basis $\{\Sigma_l\}_{l\in\Lyn Y}$ of the $\shuffle_{\phi}$-algebra). Moreover, the automorphism $\varphi_{\pi_1}$ of $(\ncp{A}{Y},\conc,\allowbreak1_{Y^*})$, mapping $y_k$ to $\pi_1y_k$ (see \eqref{pi_1}), is an isomorphism of bialgebras between $\calH_{\shuffle}(Y)$ and $\calH_{\shuffle_{\phi}}(Y)$. It follows then the linear basis $\{\Pi_w\}_{w\in Y^*}$ (resp. $\{\Sigma_w\}_{w\in Y^*}$) is image of $\{P_w\}_{w\in Y^*}$ (resp. $\{S_w\}_{w\in Y^*}$) by $\varphi_{\pi_1}$ (resp. $\check\varphi_{\pi_1}^{-1}$, where $\check\varphi_{\pi_1}$ is the adjoint of $\varphi_{\pi_1}$.
 
The above products and coproducts are extended over series by
\begin{eqnarray}
\begin{array}{rclrcl}
\shuffle:\ncs{A}{\calX}\otimes\ncs{A}{\calX}&\longrightarrow&\ncs{A}{\calX},
&\Delta_{\shuffle}:\ncs{A}{\calX}&\longrightarrow&\ncs{A}{\calX^*\otimes\calX^*},\\
\conc:\ncs{A}{\calX}\otimes\ncs{A}{\calX}&\longrightarrow&\ncs{A}{\calX},
&\Delta_{\conc}:\ncs{A}{\calX}&\longrightarrow&\ncs{A}{\calX^*},\\
{\shuffle}_{\phi}:\ncs{A}{Y}\otimes\ncs{A}{Y}&\longrightarrow&\ncs{A}{Y},
&\Delta_{{\shuffle}_{\phi}}:\ncs{A}{Y}&\longrightarrow&\ncs{A}{Y\otimes Y},
\end{array}
\end{eqnarray}
and, for any $S$ and $R\in\ncs{A}{\calX}$, by
\begin{eqnarray}
S{\shuffle}_{\phi} R=\sum_{u,v\in Y^*}\scal{S}{u}\scal{R}{v}u{\shuffle}_{\phi} v
&\mbox{and}&
\Delta_{{\shuffle}_{\phi}}S=\sum_{w\in Y^*}\scal{S}{w}\Delta_{{\shuffle}_{\phi}}w
\label{extenoverseries3},\cr
S\shuffle R=\sum_{u,v\in\calX^*}\scal{S}{u}\scal{R}{v}u\shuffle v
&\mbox{and}&
\Delta_{\shuffle}S=\sum_{w\in\calX^*}\scal{S}{w}\Delta_{\shuffle}w,\label{extenoverseries2}\\
SR=\sum_{u,v\in\calX^*\atop uv=w\in\calX^*}\scal{S}{u}\scal{R}{v}w
&\mbox{and}&
\Delta_{\conc}S=\sum_{w\in\calX^*}\scal{S}{w}\Delta_{\conc}w.\nonumber\label{extenoverseries1}
\end{eqnarray}
Note also that $\Delta_{\stuffle}S\in\ncs{A}{Y^*\otimes Y^*},\Delta_{\conc}S$ and $\Delta_{\shuffle}S\in\ncs{A}{\calX^*\otimes\calX^*}$.

Now, we are in situation to define

\begin{definition}\label{dec0}
For ${\shuffle}_{\phi}$ (resp. $\shuffle$ and $\conc$), any $S\in\ncs{A}{Y}$ (resp. $\ncs{A}{\calX}$) is
\begin{enumerate}
\item a character of $\ncp{A}{Y}$ (resp. $\ncp{A}{\calX}$) if and only if $\scal{S}{1_{Y^*}}=1_A$ (resp. $\scal{S}{1_{\calX^*}}=1_A$) and, for any $u$ and $v\in Y^*$ (resp. $\calX^*$),
\begin{eqnarray*}
\scal{S}{u{\shuffle}_{\phi} v}=\scal{S}{u}\scal{S}{v}
&\mbox{(resp.}&\scal{S}{u\shuffle v}=\scal{S}{u}\scal{S}{v}\\
&\mbox{and}&\scal{S}{uv}=\scal{S}{u}\scal{S}{v}).
\end{eqnarray*}

\item an infinitesimal character of $\ncp{A}{Y}$ (resp. $\ncp{A}{\calX}$) if and only if, for any $u$ and $v\in Y^*$ (resp. $\calX^*$),
\begin{eqnarray*}
\scal{S}{u{\shuffle}_{\phi} v}&=&\scal{S}{u}\scal{v}{1_{Y^*}}+\scal{u}{1_{Y^*}}\scal{S}{v},\cr
\mbox{(resp. }\scal{S}{u\shuffle v}&=&\scal{S}{u}\scal{v}{1_{Y^*}}+\scal{u}{1_{Y^*}}\scal{S}{v},\cr
\mbox{and }\scal{S}{uv}&=&\scal{S}{u}\scal{v}{1_{Y^*}}+\scal{u}{1_{Y^*}}\scal{S}{v}).
\end{eqnarray*}
\end{enumerate}
\end{definition}

\begin{definition}\label{gpl_p}
For $\Delta_{{\shuffle}_{\phi}}$ (resp. $\Delta_{\shuffle}$ and $\Delta_{\conc}$), a series $S$ is said to be 
\begin{enumerate}
\item grouplike if and only if $\scal{S}{1_{Y^*}}=1_{A}$ (resp. $\scal{S}{1_{\calX^*}}=1_{A}$) and
$\Delta_{{\shuffle}_{\phi}}S=S\otimes S$ (resp. $\Delta_{\shuffle}S=S\otimes S$ and $\Delta_{\conc}S=S\otimes S$).

\item primitive if and only if
$\Delta_{{\shuffle}_{\phi}}S=1_{Y^*}\otimes S+S\otimes1_{Y^*}$ (resp. $\Delta_{\shuffle}S=1_{\calX^*}\otimes S+S\otimes1_{\calX^*}$ and $\Delta_{\conc}S=1_{\calX^*}\otimes S+S\otimes1_{\calX^*})$.
\end{enumerate}
\end{definition}

\begin{proposition}\label{prim_gp}
\begin{enumerate}
\item Any series $S$ is grouplike for $\Delta_{{\shuffle}_{\phi}}$ (resp. $\Delta_{\shuffle}$ and $\Delta_{\conc}$) if and only if it is a character of $\ncp{A}{Y}$ (resp. $\ncp{A}{\calX}$) for ${\shuffle}_{\phi}$ (resp. $\shuffle$ and $\conc$).

\item Any series $S$ is primitive for $\Delta_{{\shuffle}_{\phi}}$ (resp. $\Delta_{\shuffle}$ and $\Delta_{\conc}$) if and only if it is an infinitesimal character of $\ncp{A}{Y}$ (resp. $\ncp{A}{\calX}$) for $\shuffle_{\phi}$ (resp. $\shuffle$ and $\conc$). 

\item\label{ree} If $\scal{S}{1_{\calX^*}}=1$ then $S$ is grouplike for $\shuffle_{\phi}$ (resp. $\shuffle$ and $\conc$) if and only if $\log S$ is primitive for $\Delta_{{\shuffle}_{\phi}}$ (resp. $\Delta_{\shuffle}$).

\item The set of grouplike series, denoted by $\calG_{{\shuffle}_{\phi}}^Y$ (resp. $\calG_{\shuffle}^{\calX}$ and $\calG_{\conc}^{\calX}$), is a group.

\item The set of primitive series, denoted by $\calP_{{\shuffle}_{\phi}}^Y$ (resp. $\calP_{\shuffle}^{\calX}$ and $\calP_{\conc}^{\calX}$) is a Lie algebra.
\end{enumerate}
\end{proposition}

\begin{proof}
These facts are classical in theory of Hopf algebras \cite{abe,CartierPatras,VJM,reutenauer}.
\end{proof}

\section{Various characterizations of representative series}
Representative series are representative functions on the free monoid. Indeed,

\begin{definition}\label{dec1}
Let $S\in\ncs{A}{\calX}$ (resp. $\ncp{A}{\calX}$) and $P\in\ncp{A}{\calX}$ (resp. $\ncs{A}{\calX}$).
Then the {\it left} and the {\it right} \textit{shifts}\footnote{These are called {\it residuals} and extend shifts of functions in harmonic analysis \cite{jacob}.
In terms of representative functions, these are the {\it left} and {\it right} \textit{translates} \cite{abe,CartierPatras}.} of $S$ by $P$, $P\rd S$ and {$S\rg P$}, are defined, for any $w\in\calX^*$, by
$\scal{{P\rd S}}{w}=\scal{S}{wP}$ and $\scal{{S\rg P}}{w}=\scal{S}{Pw}$.
\end{definition}

\begin{remark}
The shifts operators are associative and mutually commute, \textit{i.e.}
$S\rg(P\rg R)=(S\rg P)\rg R,
P\rd(R\rd S)=(P\rd R)\rd S,
(P\rg S)\rd R=P\rg(S\rd R)$
and then one has $x\rd(wy)=(yw)\rg x=\delta_{x,y}w$, for
$x,y\in\calX$ and $w\in\calX^*$.
\end{remark}

\begin{definition}[\cite{berstel}]\label{rational}
The series $S$ is rational if it belongs to the smallest algebraic closure by rational operations (conatenation, addition,  Kleene star) containing $\ncp{A}{\calX}$.
The $A$-module of rational series is denoted by ${\ncs{A^{\mathrm{rat}}}\calX}$.
\end{definition}

\begin{definition}\label{dec1bis}
Let $S\in\ncs{K}{\calX}$ (resp. $\ncs{K}{Y}$).
The Sweedler's dual $\calH_{\shuffle}^{\circ}(\calX)$ (resp. $\calH_{{\shuffle}_{\phi}}^{\circ}(Y)$) of $\calH_{\shuffle}(\calX)$ (resp. $\calH_{{\shuffle}_{\phi}}(Y)$) is defined, with a family $\{G_i,D_i\}_{i\in I}$ of series in $\calH_{\shuffle}^{\circ}(\calX)$ (resp. $\calH_{{\shuffle}_{\phi}}^{\circ}(Y)$) and finite $I$, by
\begin{eqnarray*}
S\in\calH_{\shuffle}^{\circ}(\calX)\mbox{ (resp. $\calH_{{\shuffle}_{\phi}}^{\circ}(Y)$)}
&\iff&\Delta_{\conc}(S)=\sum_{i\in I}G_i\otimes D_i.
\end{eqnarray*}
\end{definition}

\begin{remark}\label{dec2}
Let $S\in\ncs{A}{\calX}$ and suppose that there is some finite set $I$  and a double family $\{G_i,D_i\}_{i\in I}$ of series in $\ncs{A}{\calX}$ such that, using $\Delta_{\conc}$,
\begin{eqnarray*}
\Delta_{\conc}(S)=\sum_{i\in I}G_i\otimes D_i.
\end{eqnarray*}
Then, for any $v\in\calX^*$ and $i\in I$, putting $G'_i=G_i\rg v$ and $D'_i=v\rd D_i$, one has
\begin{enumerate}
\item $\Delta_{\conc}(S\rg v)=\sum\limits_{i\in I}G'_i\otimes D_i$ and
$\Delta_{\conc}(v\rd S)=\sum\limits_{i\in I}G_i\otimes D'_i$.

\item $\{S\rg v\}_{v\in\calX^*}$ (resp. $\{v\rd S\}_{v\in \calX^*}$) lie in a finitely generated shift-invariant $A$-module if and only if $\{G_i\rg v\}_{v\in \calX^*}$ (resp. $\{v\rd D_i\}_{v\in \calX^*}$) does (for $i\in I$).

\item If $S\in\calH_{\shuffle}^{\circ}(\calX)$ then $v\rd S$ and $S\rg v\in\calH_{\shuffle}^{\circ}(\calX)$ ($v\in\calX^*$).

\item $\calH_{\shuffle}(\calX)$ and $\calH_{{\shuffle}_{\phi}}(Y)$ are graded while $\calH_{\shuffle}^{\circ}(\calX)$ and $\calH_{{\shuffle}_{\phi}}^{\circ}(Y)$ are not.
\end{enumerate}
\end{remark}

\begin{theorem}\label{KS}
A series $S$ is rational if and only if one of the following assertions holds
\begin{enumerate}
\item The shifts $\{S\trl w\}_{w\in\calX^*}$ (resp. $\{w\rd S\}_{w\in \calX^*}$) lie in a finitely generated shift-invariant $A$-module \cite{jacob-these}.

\item There is $n\in{\mathbb N}$ and a linear representation $(\nu,\mu,\eta)$ of rank $n$ of $S$ such that $\scal{S}{w}=\nu\mu(w)\eta$ (for $w\in\calX^*$), where $\nu\in{\cal M}_{n,1}(A),\eta\in{\cal M}_{1,n}(A)$ and $\mu:\calX^*\longrightarrow{\cal M}_{n,n}(A)$ (Kleene-Sch\"utzenberger theorem) \cite{berstel}.
\end{enumerate}
\end{theorem}

\begin{definition}\label{nil_sol}
\begin{enumerate}
\item Let $\calL$ be the Lie algebra. Then $\calL$ is \textit{nilpotent} (resp. \textit{solvable}) if and only if there exists an integer $k\ge1$ such that the sequence $\{\calL^n\}_{n\ge1}$ (resp. $\{\calL^{(n)}\}_{n\ge1}$), defined recursively by
$\calL^1=\calL,\calL^{n+1}=[\calL,\calL^n]$
(resp. $\calL^{(1)}=\calL,\allowbreak
\calL^{(n+1)}=[\calL^{(n)},\calL^{(n)}])$,
satisfies $\calL^{k+1}=\{0\}$ (resp. $\calL^{(k+1)}=\{0\}$).

\item Let $(\nu,\mu,\eta)$ be a linear representation of $S\in\ncs{A^{\mathrm{rat}}}{\calX}$.
One defines
\begin{enumerate}
\item the Lie algebra generated by $\{\mu(x)\}_{x\calX}$ and denoted by $\calL(\mu)$,
\item the function on monoid $\begin{array}{ccc}
M:\calX^*\longrightarrow M_{n,n}(\ncs{A}{\calX}),&&w\longmapsto\mu(w)w.
\end{array}$
\end{enumerate}
\end{enumerate}
\end{definition}

\begin{proposition}\label{linearrepresentation}
The module ${\ncs{A^{\mathrm{rat}}}\calX}$ (resp. ${\ncs{A^{\mathrm{rat}}}{Y}}$)
is closed by $\shuffle$ (resp. ${\shuffle}_{\phi}$). Moreover, for any $i=1,2$, let
$R_i\in{\ncs{A^{\mathrm{rat}}}\calX}$ and $(\nu_i,\mu_i,\eta_i)$ be its
representation of rank $n_i$. Then the linear representation
$$\begin{array}{rccl}
\mbox{that of}&R_i^*&\mbox{is}&\Big(\begin{pmatrix}0&1\end{pmatrix},
\left\{\begin{pmatrix}\mu_i(x)+\eta_i\nu_i\mu_i(x)&0\cr\nu_i\eta_i&0\end{pmatrix}\right\}_{x\in\calX},
\begin{pmatrix}\eta_i\cr1\end{pmatrix}\Big),\cr
\mbox{that of}&R_1+R_2&\mbox{is}&\Big(\begin{pmatrix}\nu_1&\nu_2\end{pmatrix},
\left\{\begin{pmatrix}\mu_1(x)&{\bf 0}\cr{\bf 0}&\mu_2(x)\end{pmatrix}\right\}_{x\in\calX},
\begin{pmatrix}\eta_1\cr\eta_2\end{pmatrix}\Big),\cr
\mbox{that of}&R_1R_2\quad&\mbox{is}&\Big(\begin{pmatrix}\nu_1&0\end{pmatrix},
\left\{\begin{pmatrix}\mu_1(x)&\eta_1\nu_2\mu_2(x)\cr0&\mu_2(x)\end{pmatrix}\right\}_{x\in\calX},
\begin{pmatrix}\eta_1\mu_2\eta_2\cr\eta_2\end{pmatrix}\Big),\cr
\mbox{that of}&R_1\shuffle R_2&\mbox{is}&(\nu_1\otimes\nu_2,\{\mu_1(x)\otimes\mathrm{I}_{n_2}
+\mathrm{I}_{n_1}\otimes\mu_2(x)\}_{x\in\calX},\eta_1\otimes\eta_2),\\
\mbox{that of}&R_1{\shuffle}_{\phi} R_2&\mbox{is}&(\nu_1\otimes\nu_2,
\{\mu_1(y_k)\otimes\mathrm{I}_{n_2}+\mathrm{I}_{n_1}\otimes\mu_2(y_k)\cr
&&&+\sum\limits_{i+j=k}\gamma_{i,j}^k\mu_1(y_i)\otimes\mu_2(y_j)\}_{k\ge1},\eta_1\otimes\eta_2).
\end{array}$$
\end{proposition}

\begin{proof}
The constructions of linear representations are classical \cite{jacob} (the representations of $R_1\shuffle R_2$ and $R_1{\shuffle}_{\phi} R_2$ base on coproducts and tensor products of representations). Only the last one is new.
\end{proof}

\begin{corollary}[Factorization and decomposition, \cite{CM}]\label{factorized}
Let $S\in\ncs{A^{\mathrm{rat}}}{\calX}$ of a linear representation $(\nu,\mu,\eta)$. Then, with Notations in Definition \ref{nil_sol},
\begin{enumerate}
\item $S=\nu M(\calX^*)\eta$ and
\begin{eqnarray*}
M(\calX^*)=\prod_{l\in\Lyn\calX}^{\searrow}e^{\mu(P_ l)S_l}
&\Big(\mbox{resp.}&
M(Y^*)=\prod_{l\in\Lyn Y}^{\searrow}e^{\mu(\Pi_ l)\Sigma_l}\Big).
\end{eqnarray*}

\item If $\{M(x)\}_{x\in\calX}$ are upper triangular  then $S=\nu((D(\calX^*)N(\calX))^*D(\calX^*)\eta$, where $N(\calX)$ (resp. $D(\calX)$) is a strictly upper triangular (resp. diagonal) matrix such that $M(\calX)=N(\calX)+D(\calX)$. Moreover, $D(\calX^*)$ is diagonal and there is a positive interger $k$ such that $D(\calX^*)N(\calX)$ is nilpotent of order $k$ and then
$S=\nu(I_n+D(\calX^*)N(\calX^*)+\ldots+(D(\calX^*)N(\calX^*))^k)D(\calX^*)\eta$.
\end{enumerate}
\end{corollary}

\begin{proof}
\begin{enumerate}
\item One has  $M(\calX^*)=(\mathrm{Id}\otimes\mu)\calD_{\calX}$ (resp. $M(Y^*)=(\mathrm{Id}\otimes\mu)\calD_Y$).

\item By Lazard factorization \cite{reutenauer}, it follows the expected results.
\end{enumerate}
\end{proof}

\begin{proposition}[\cite{CM}]\label{PQ}
With Notations in Theorem \ref{KS}, let $G_i$ (resp $D_i$) belong to $\ncs{A^{\mathrm{rat}}}{\calX}$ admitting $(\nu,\mu,e_i)$ (resp. $({}^te_i,\mu,\eta)$) as linear representation of rank $n$ ($1\le i\le n$), where 
$e_i\in M_{1,n}(A)$ and $\underset{\hskip1cm\uparrow i}{{}^te_i=
{\begin{matrix}(0&\ldots&0&1&0&\ldots&0)\end{matrix}}}$. Then
\begin{eqnarray*}
\\[-5mm]
\Delta_{\conc}S=\sum_{1\le i\le n}G_i\otimes D_i.
\end{eqnarray*}
\end{proposition}

\begin{proof}
The proof given in \cite{CM} is formulated, for any $u$ and $v\in\calX^*$, as follows
\begin{eqnarray*}
\scal{S}{uv}=\beta\mu(u)\mu(v)\eta=\sum_{1\le i\le n}(\nu\mu(u)e_i)({}^te_i\mu(v)\eta)=\sum_{1\le i\le n}\scal{G_i}{u}\scal{D_i}{v},\\
\scal{\Delta_{\tt conc}S}{u\otimes v}=\scal{S}{uv}=\sum_{1\le i\le n}\scal{G_i}{u}\scal{D_i}{v}=\sum_{1\le i\le n}\scal{G_i\otimes D_i}{u\otimes v}.
\end{eqnarray*}
Extending by linearity and then by $\Delta_{\conc}$, it follows the expected result, since
\begin{eqnarray*}
\forall P,Q\in\ncp{A}{\calX},&&
\scal{S}{PQ}=\sum_{1\le i\le n}\scal{G_i}{P}\scal{D_i}{Q}.
\end{eqnarray*}
\end{proof}

\begin{corollary}\label{Sweedler}
\begin{enumerate}
\item With Notations of Definition \ref{dec1bis}, Propositions \ref{linearrepresentation}--\ref{PQ}, one has
\begin{enumerate}
\item ${\ncs{A^{\mathrm{rat}}}\calX}$ is an unital $A$-algebra with respected to one of $\{\conc,\shuffle,{\shuffle}_{\phi}\}$.

\item The following criterion characterizes rational (or representative) series
\begin{eqnarray*}
S\in\ncs{A^{\mathrm{rat}}}{\calX}\iff
\Delta_{\conc}S=\sum_{i\in I_f}G_i\otimes D_i.
\end{eqnarray*}
\end{enumerate}

\item $\calH_{\shuffle}^{\circ}(\calX)$ (resp. $\calH_{{\shuffle}_{\phi}}^{\circ}(Y)$) is isomorphic to $(\ncs{K^{\mathrm{rat}}}{\calX},\shuffle,1_{\calX^*},\Delta_{\conc})$ (resp. $(\ncs{K^{\mathrm{rat}}}{Y},{\shuffle}_{\phi},1_{Y^*},\Delta_{\conc})$) of rational (or representative) series.
\end{enumerate}
\end{corollary}

\begin{proof}
\begin{enumerate}
\item These are consequences of Propositions \ref{linearrepresentation}--\ref{PQ}, respectively.

\item Previous criterion yields the expected results for $A=K$ (see Remark \ref{dec2}). 
\end{enumerate}
\end{proof}

\end{document}